\documentclass[a4paper,11pt,reqno,noindent]{amsart}
\usepackage[centertags]{amsmath} 
\usepackage{amsfonts,amssymb,amsthm}
\usepackage{mathtools}

\usepackage{graphicx}
\usepackage{dsfont}
\usepackage[english]{babel} \usepackage{newlfont} \usepackage{color}
\usepackage[body={15cm,21.5cm},centering]{geometry} \usepackage{fancyhdr}
 \usepackage{esint} 

 \newcommand{\Tr}{\mathrm{Tr}\,}

\usepackage[active]{srcltx}
\usepackage{mathrsfs}

\newtheorem{theorem}{Theorem}[section] 
\newtheorem{lemma}[theorem]{Lemma}
\newtheorem{proposition}[theorem]{Proposition}

\newtheorem{algo}[theorem]{Algorithm}

\theoremstyle{definition} 
\newtheorem{remark}[theorem]{Remark}

 \newcommand{\supp}{\operatorname{supp}}

\newcommand{\e}{\varepsilon}

\newcommand{\R}{\mathbb{R}} 
\newcommand{\N}{\mathbb{N}} 
 
\renewcommand{\d}{\mathrm{d}} \renewcommand{\L}{\mathbb{L}}
 
 \newcommand{\id}{\mathrm{Id}}

\newcommand{\dL}{\mathrm{d}\mathcal{L}} 
 
\renewcommand{\L}{{\mathscr L}} \newcommand\wto{\rightharpoonup}

\newcommand\wsto{\stackrel{*}{\rightharpoonup}}

\newcommand{\PA}{\mathrm{PA}}
\newcommand{\PAM}{\mathrm{PAM}}
\newcommand{\MBV}{\mathrm{MBV}}

\newcommand\blfootnote[1]{%
  \begingroup
  \renewcommand\thefootnote{}\footnote{#1}%
  \addtocounter{footnote}{-1}%
  \endgroup
}

\makeatletter
\def\widebreve{\mathpalette\wide@breve}
\def\wide@breve#1#2{\sbox\z@{$#1#2$}%
     \mathop{\vbox{\m@th\ialign{##\crcr
\kern0.08em\brevefill#1{0.8\wd\z@}\crcr\noalign{\nointerlineskip}%
                    $\hss#1#2\hss$\crcr}}}\limits}
\def\brevefill#1#2{$\m@th\sbox\tw@{$#1($}%
  \hss\resizebox{#2}{\wd\tw@}{\rotatebox[origin=c]{90}{\upshape(}}\hss$}
\makeatletter

\begin{document}

\title[Asymptotic meshes from $r$-variational adaptation]{Asymptotic meshes from $r$-variational adaptation methods for static problems in one dimension}

\author[D. Hun]{Darith Hun}
\address[]{Université catholique de Louvain, Belgium}
\email[Darith Hun]{darith.hun@uclouvain.be}
\author[N. Moës]{Nicolas Moës}
\address[]{Université catholique de Louvain, Belgium}
\email[Nicolas Moës]{nicolas.moes@uclouvain.be}
\author[H. Olbermann]{Heiner Olbermann}
\address[]{Université catholique de Louvain, Belgium}
\email[Heiner Olbermann]{heiner.olbermann@uclouvain.be}


\renewcommand{\chi}{\mathds{1}}

\maketitle

\begin{abstract}
 We consider the minimization of integral functionals in one dimension and their approximation by $r$-adaptive finite elements. Including the grid of the FEM approximation as a variable in the minimization, we are able to show that the optimal grid configurations have a well-defined limit when the number of nodes in the grid is being sent to infinity. This is  done by showing that the suitably renormalized energy functionals possess a limit  in the sense of $\Gamma$-convergence. We provide numerical examples showing the closeness of the optimal asymptotic mesh obtained as a minimizer of the $\Gamma$-limit to the optimal finite meshes.  
\end{abstract}

{\small
\noindent 
}

\blfootnote{ D.H. and N.M. acknowledge support for this project through funding from the European Research Council (ERC) under the European Union’s Horizon research and innovation program (Grant agreement No. 101 071 255).}

\section{Introduction}

\subsection{Motivation}
In the study of minimization problems defined over suitable sets of  functions $X$ ,
\begin{equation}\label{eq:22}
    \inf_{u \in X} E(u),
\end{equation}
finite element methods are widely employed to approximate solutions. These methods involve restricting minimization to finite-dimensional affine subspaces $X_h$ of $X$, often defined by a set of parameters that govern the mesh. An adaptive finite element method includes optimizing the  choice of the finite-dimensional subspace $X_h$ by appropriately selecting these parameters. Here we will focus on $r$-adaptive methods for simplicial meshes, for which the topology of the mesh and the reference elements are fixed, and only the position of the mesh points changes.

\medskip

For linear problems, that is, variational problems with linear Euler-Lagrange equations, the discrepancy between the exact solution $u$ and the discrete approximation $u_h$ can often be measured in terms of an "energy norm" $\|u_h - u\|_E$. Minimizing this quantity provides a practical criterion for optimizing the grid. However, such an approach cannot be generalized in a straightforward manner to nonlinear problems. 

\medskip

In the present article, we will directly consider the minimality of the energy functional $E$ itself as the  criterion for optimizing the grid. In this framework, the parameters defining the subspace $X_h$ (i.e., the positions of the mesh points) are treated as additional variables in the minimization problem. This perspective, which introduces an intrinsic coupling between the solution and the discretization, can be traced back to early works such as \cite{felippa1976optimization,mcneice1973optimization}. More recently, the concept has been further developed in the context of \emph{configurational equilibrium}, a framework for mesh optimization that has gained attention in both engineering and mathematical communities (see, e.g., \cite{thoutireddy2004variational,mueller2002material}).  

\medskip

A natural candidate for the subspaces $X_h$ is the class of continuous piecewise affine functions defined on simplicial meshes $\mathcal{T}_h$ of a given fixed polygonal domain $\Omega$. Here, each function in $X_h$ is affine on each simplex in $\mathcal{T}_h$. The minimization process will then include minimization over the position of the mesh points of $\mathcal T_h$, the topology of the mesh being fixed. 

\medskip

The primary objective of this article is to investigate whether the optimal meshes obtained through this approach converge, in a suitable sense, to an asymptotically optimal mesh as the number of simplices tends to infinity. In other words, on the level of equations, we are interested in passing to the limit of the  configurational equilibrium equations. Focusing exclusively on the one-dimensional case, and assuming suitable regularity and coercivity properties, we will demonstrate that this convergence does even occur on a functional level in the sense of $\Gamma$-convergence, which implies in particular the convergence of solutions of configurational equilibrium. We thus obtain a rigorous and theoretically satisfactory analysis of the problem.




\medskip

From our analysis and the result for the limit functional, it is relatively straightforward to conjecture a similar behavior for piecewise affine approximations of minimizers of variational problems in dimensions larger than one, $d>1$. If one wishes to establish suitable compactness results that are crucial for a rigorous analysis in the sense of $\Gamma$-convergence, this is, however, much more challenging than in the present case $d=1$, and will be addressed elsewhere.

\medskip

Our results share some characteristics with the analysis of adaptive meshes from \cite{courty2006continuous,loseille2011continuous}, where a variational problem for the mesh that best represents a twice continuously differentiable function is considered.

\subsection{Outline of the article} The article is structured as follows: In Section \ref{sec:sett-stat-main}, we state and prove  the main theorem, Theorem \ref{thm:Fjgamma}. In Section \ref{sec:tools-notat-auxil}, we present some technical  preliminaries. The proof of lower and upper bound that constitute $\Gamma$-convergence will be given in Sections \ref{sec:proof-comp-lower} and \ref{sec:proof-upper-bound} respectively.  In Section \ref{sec:examples}, we provide numerical illustrations of our results. The proof of a technical lemma used in the lower bound is proved in the appendix. 

\subsection{Notation}
The Lebesgue measure on $\R$ is denoted by $\L^1$; integration of an integrable function $f:\R\to \R$ with respect to it is denoted by $\int f\d \L^1$, or by $\int f(x)\d x$. Let $I=[a,b]\subset \R$.  The weak convergence of a sequence $(u_j)_{j\in\N}$ to $u$ in $L^p(I)$  is denoted by $u_j\wto u$. We will use the following, slightly non-standard, definition of $BV$ functions on a closed interval with fixed boundary values: Supposing $\alpha,\beta\in\R$, we denote by $u\in BV^{\alpha,\beta}(I)$ the restrictions to $I$ of functions $\tilde u\in BV_{\mathrm{loc.}}(\R)$ with $\tilde u(x)=\alpha$ for $x<a$, $\tilde u(x)=\beta$ for $x>b$. In particular, for $u\in BV^{\alpha,\beta}(I)$ whose extension to $\R$ as above is denoted by $\tilde u$,  there exists  a signed measure $[Du]\in \mathcal M(\R)$ with support in $I$ such that
\[
\int_\R \varphi(x)\cdot\d[D u](x)=\int_\R  \tilde u(x)\varphi'(x) \d x \quad \forall \varphi\in C^1_c(\R)\,.
\]
Whenever we deal with $BV$ functions in the present paper, these will be functions in  $BV^{a,b}(I)$; which is why we drop the indices and simply write $BV(I)$ from now on. 
For a sequence $(u_j)_{j\in\N}\subset BV(I)$ and  $u\in BV(I)$ such that $u_j\to u$ in $L^1(I)$ and 
\[
\lim_{j\to\infty}\int_I \varphi(x)\cdot \d[Du_j](x)=\int_I \varphi(x)\cdot \d[Du](x)\quad \forall \varphi\in C^0(I)\,,
\]
we will write $u_j\wsto u$ in $BV(I)$. We will use the notation $u'\equiv[Du]$ for $u\in BV(I)$. If $u\in BV(I)$, then by Lebesgue's decomposition theorem, we may write
\[
u'=\frac{\d u'}{\d\L^1}\L^1+ (u')_s\,,
\]
where $(u')_s$ is singular with respect to Lebesgue measure.
 The characteristic function of a set $A$ is denoted by  $\mathds{1}_A$. The scalar product between matrices $A,B\in\R^{k\times k}$ is defined by  $A:B=\Tr(A^TB)$, where $\Tr$ denotes the the trace, and the upper index $T$ denotes the transpose.  The symbol ``$C$'' will be used as follows: An inequality such as $f\leq C g$ has to be read as ``there exists a numerical constant $C>0$ such that $f\leq Cg$''. If the constant depends on other quantities $a,b,\dots$, then we write $C(a,b,\dots)$. We also write $f\lesssim g$ instead of $f\leq Cg$, and $f\simeq g$ for $C^{-1}g\leq f\leq C g$.

\section{Setting and statement of the main theorem}

\label{sec:sett-stat-main}
Let $I=[a,b]\subset\R$, $\theta>0$, $N\in\N$, $1<q,r<\infty$. Consider a  Lagrangian
\[
  \begin{split}
  \mathcal L: I\times\R^N\times \R^N&\to\R\\
  (x,z,p)&\mapsto \mathcal L(x,z,p)
\end{split}
\]
satisfying
\[
  \begin{split}
    |\mathcal L(x,z,p)|&\leq C(1+|z|^r+|p|^q)\,.
      \end{split}
    \]
The symbols $z,p$ will always have the meaning of the respective arguments of $\mathcal L$.
We will denote by   $\nabla_{(z,p)}$  the  $2N$-dimensional gradient with respect to the variables $z,p\in\R^N$,
\[
\nabla_{(z,p)} \mathcal L(x,z,p)=(\partial_{z_1}\mathcal L,\partial_{z_2}\mathcal L,\dots,\partial_{z_N}\mathcal L,\partial_{p_1}\mathcal L,\dots,\partial_{p_N}\mathcal L)^T\,.
\]
The Hessian $\nabla_{(z,p)}^2\mathcal L$ will denote the ${2N\times 2N}$ matrix  containing the respective partial  second derivatives with respect to $z$ and $p$, while $\nabla^2_p\mathcal L$ denotes the $N\times N$ matrix containing the second partial derivatives with respect to $p$.


The action functional associated to $\mathcal L$ reads
\[
\mathcal F(u)=\int_a^b \mathcal L(x,u(x),u'(x))\d x\,.
\]

Let $U_a,U_b\in\R^N$, and $\mathcal A\equiv \mathcal A(U_a,U_b)$ given by
\[
  \mathcal A=\{u\in W^{1,q}(I;\R^N):u(a)=U_a,\,u(b)=U_b\}\,.
  \]

The existence and regularity of a minimizer $u_*$ of the variational problem 
  \begin{equation}\label{eq:43}
\inf_{u\in\mathcal A}\mathcal F(u)
\end{equation}
under standard assumptions such as convexity of the map $p\mapsto \mathcal L(x,z,p)$ is guaranteed by well-known theorems, 
 see e.g.~Theorems 3.7 and 4.1 in \cite{buttazzo1998one}.
We will make the 
following technical assumption   on the stability  of the minimizer $u_*$:
 \begin{itemize}
\item[(A1)] $\mathcal L\in C^2(I\times\R^N\times\R^N)$,  and there exists $\theta>0$ such that for  every $x\in I$,  
  \begin{equation}\label{eq:5}
    \begin{split}
\int_0^1&(1-\tau)\nabla^2_{(z,p)}\mathcal L(x,u_*(x)+\tau v,u_*'(x)+\tau w)\d\tau\\
&\geq \theta\left(
  \begin{array}{cc}
    \id_{N\times N} & 0\\ 0 & 0 
  \end{array}
\right)  \quad \forall v,w\in\R^N\,\forall x\in I\,.
\end{split}
\end{equation}
\end{itemize}

\begin{remark}
  \begin{itemize}
\item[(i)] The assumption (A1) implies in particular that the minimizer $u_*$ is unique. 
  \item[(ii)] Property (A1) depends on the Lagrangian and the
    minimizer $u_*$.  An example of a pair that satisfies the above
    assumption is given by
\[
\mathcal L(x,z,p)= A(x)k(p)+B(x)l(z)+f(x)\cdot z\,,
\]
where $A(x)\geq  \theta_1>0$ and $B(x)> 0$ for all $x\in I$, $k,l:\R^N\to\R$ are convex and $C^2$,
$f\in C^2(I;\R^N)$,  provided that for some $\theta_2>0$ the minimizer satisfies 
  \begin{equation}\label{eq:25}
    \nabla^2_pk(u_*'(x))\geq \theta_2\id_{N\times N}\quad\forall x\in I\,.
  \end{equation}
\item[(iii)] An interesting and very simple example is given by the Dirichlet problem, defined by 
\[
\mathcal L(x,z,p)=\frac12 |p|^2+f(x)\cdot z\,.
\]
This is just  a special case of (i), but in this case,  the condition \eqref{eq:25} on  the unique minimizer $u_*$ is trivially fulfilled. Our numerical examples in Section \ref{sec:examples} will consider optimal meshes for this problem, with $N=1$.
\item[(iv)] It is possible to relax condition (A1) to allow for Lagrangians that are not necessarily $C^2$, while preserving the identity of Taylor's Theorem for $C^2$ functions with an integral rest (see \eqref{eq:19} below) in a suitable sense, as well as a  sufficient coercivity of the second variation  $g\mapsto \delta^2 \mathcal F(u_*,g)$ (see equation \eqref{eq:6} below) for our argument to hold with some modifications. We have refrained from admitting this slightly more general setting in order to keep the exposition short and transparent. 
\end{itemize}

\end{remark}

We will consider the approximation of the unique minimizer $u_*$ via adaptive piecewise affine finite elements. For $n\in \N$ and $i\in \{0,\dots,n-1\}$, we write
  \begin{equation}\label{eq:23}
\bar X_i^n=\left[a+\frac{i(b-a)}{n},a+\frac{(i+1)(b-a)}{n}\right)\,.
\end{equation}
These should be thought of the cells in a reference configuration.
Let $\bar x_i^n$ denote the midpoint of $\bar X_i^n$, 
\[
\bar x_i^n=a+\frac{(i+1/2)(b-a)}{n}\,.
\]
We define continuous piecewise affine functions corresponding to ${\bar X}_i^n$, 
\[
\PA^n(I)=\{u\in C^0(I;\R^N): u|_{{\bar X}_i^n} \text{ affine for } i=0,\dots,n-1\}\,.
\]
Furthermore, we define a class of   monotone continuous piecewise affine  functions $[a,b]\to[a,b]$, 
\[
  \begin{split}
\PAM^n(I):=&\{\bar y\in C^0(I): \bar y|_{{\bar X}_i^n} \text{ affine for } i=0,\dots,n-1,\\
&\quad \bar y'\geq 0 \text{ almost everywhere, } \bar y(a)=a, \bar y(b)=b\}\,.
\end{split}
\]
For $\bar y\in \PAM^n(I)$, the adapted cells are given by $X_i^n=\bar y_n(\bar X_i^n)$, $i=0,\dots,n-1$. In this way, we can capture the information of an adapted mesh in the choice of some $\bar y\in \MBV(I)$.

\medskip
 
Let us denote by $\MBV(I)$ the space of 
 monotonous functions $w\in BV(I)$ with $w(a)\geq a$, $w(b)\leq b$. For $w\in \MBV(I)$, we may consider the $BV$ function whose graph (understood as the boundary of the subgraph) we obtain by reflection of the graph of $w$ across the diagonal $\{x_1=x_2\}$ in $\R^2$, and denote it by $w^{-1}$, which is again in $\MBV(I)$. More precisely, we define $w^{-1}$ by requiring
\[
w^{-1}(x_2)< x_1 \quad \Leftrightarrow \quad x_2< w(x_1)\quad \text{ for all } (x_1,x_2)\in [a,b]^2\,.
\]
Clearly $\PAM^n(I)\subset \MBV(I)$.

\medskip

We define the functionals $\mathcal F_n:L^1(I;\R^{N+1})\to[0,\infty]$  by setting
\[
\mathcal F_n(u,y)=\begin{cases} n^2\left(\mathcal F(u)-\mathcal F(u_*)\right)  & \text{ if }y^{-1}\in \PAM^n(I),\,u\circ y^{-1}\in \PA^n(I)\cap \mathcal A\\
+\infty & \text{ else.}\end{cases}
\]
Furthermore we 
set
\begin{equation}\label{eq:application}
{\mathcal F}^*(g,y)=\frac12\delta^2 \mathcal F(u_*,g)
+\frac{(b-a)^2}{24}\int_I\nabla_p^2\mathcal L(x,u_*,u_*'):u_*''\otimes u_*''\left|\frac{\d y'}{\d \L^1}\right|^{-2}\d x
\end{equation}
if $y\in \MBV(I)$, $g\in W^{1,2}_0(I)$, and $\mathcal F^*=+\infty$ otherwise. The integral above has to be understood in the following way: We define $u_*''(x)\otimes u_*''(x)\left|\frac{\d y'}{\d \L^1}\right|^{-2}$ as the measurable function satisfying $u_*''(x)\otimes u_*''(x)\left|\frac{\d y'}{\d \L^1}(x)\right|^{-2}\equiv 0$ if $u_*''(x)=0=\frac{\d y'}{\d \L^1}(x)$,  and $u_*''(x)\otimes u_*''(x)\left|\frac{\d y'}{\d \L^1}\right|^{-2}\equiv +\infty$ if $u_*''(x)\neq 0=\frac{\d y'}{\d\L^1}(x)$.
The  integral    is understood  to be equal to $+\infty$ if the integrand  is not in $L^1$. Finally
 $\delta^2\mathcal F$ denotes the second variation of $\mathcal F$ given by 
  \begin{equation}\label{eq:6}
    \delta^2 \mathcal F(u_*,g)=\int_I(g(x)^T,g'(x)^T)\nabla^2_{(z,p)}\mathcal L(x,u_*(x),u_*'(x))\left(
      \begin{array}{c}
        g(x)\\ g'(x)
      \end{array}
    \right)\d x\,.
  \end{equation}
Our main result is the $\Gamma$-convergence $\mathcal F_n\to\mathcal F^*$:
\begin{theorem}\label{thm:Fjgamma}
Let $\mathcal L:I\times \R^N\times\R^N\to\R$ satisfy condition (A1).
  \begin{itemize}
  \item[(o)] Suppose $(u_n,y_n)\in L^1(I;\R^{N+1})$ such that
    $\sup_n\mathcal F_n(u_n,y_n)<\infty$. Then there exists  a subsequence
    (which we do not relabel) and $g\in W^{1,2}_0(I)$, $y\in MBV(I)$  such that
\[
n(u_n-u_*)\wto g\quad\text{ in } W^{1,2}_0(I)\,,\qquad y_n\wsto y \quad\text{ in } BV(I)\,.
\]
\item[(i)] Suppose $n(u_n-u_*)\wto g$ and $y_n\wsto y$ as in the previous point. Then 
  \begin{equation}\label{eq:20}
  \begin{split}
\liminf_{n\to\infty} \mathcal F_n(u_n,y_n)\geq \mathcal F^*(g,y)\,.
\end{split}
\end{equation}

\item[(ii)] Suppose $g\in W^{1,2}_0(I)$, $y\in \MBV(I)$. Then there exists a  sequence  $(u_{n},y_{n})_{n\in\N}\subset \PA^{n}(I)\times \PAM^{n}(I)$ such that $n(u_{n}-u_*)\wto g$ in $W^{1,2}(I)$, $y_{n}\wsto y$ in $BV(I)$ and
\[
  \begin{split}
\limsup_{n\to \infty} \mathcal F_{n}(u_{n},y_{n})&=\mathcal F^*(g,y)\,.
\end{split}
\]
\end{itemize}
\end{theorem}

\begin{remark}
  \begin{itemize}
  \item[(i)] Our result implies that for any sequence of minimizers $(u_n,y_n)$ of $\mathcal F_n$ (which, in other words, is just  a minimizer of $\mathcal F$ within the class of continuous functions that are piecewise affine on $n$ pieces whose position is not fixed), there exists a subsequence (which we do not relabel) and $g\in W^{1,2}_0(I)$, $y\in \MBV(I)$, such that $n(u_n-u_*)\wto g$ in $W^{1,2}$, $y_n\wsto y$ in $BV$, and $(g,y)$ is a minimizer of $\mathcal F^*$. For a proof of this fundamental fact in the theory of $\Gamma$-convergence, see \cite{braides2002gamma,dal2012introduction}.
\item[(ii)] In a sense, we rediscover an energy norm to be minimized in the limit $n\to \infty$, suitable  for nonlinear problems. 
  \end{itemize}
\end{remark}

\section{Tools, notation and preparatory lemmata}

\label{sec:tools-notat-auxil}






\subsection{Auxiliary notation and lemmata}
\label{sec:auxiliary-notation}
In our proof, barred symbols will always be associated with the ``reference configuration'' that corresponds to a regularly spaced grid: We have already defined $\bar X_i^n=[a+i\frac{b-a}{n},a+(i+1)\frac{b-a}{n})$. The inverse of a function $\bar y_n\in\PAM^n(I)$ will be denoted by $y_n$, and it is   increasing and affine on each $X_i^n$, $i=0,\dots,n-1$, where $X_i^n=\bar y_n(\bar X_i^n)$. The midpoint of $X_i^n$ is given by $x_i^n=\bar y_n(\bar x_i^n)$.

\medskip

In the upcoming proofs of upper and lower bound, we will study the functions
\[
g_n=n(u_n-u_*)\,,
\]
where $u_n$ is continuous and piecewise affine on each $X_i^n$, $i=0,\dots,n-1$ (i.e., it is of the form $u\circ y^{-1}$ for some $u\in \PA^n(I)$ and $y\in \PAM^n(I)$).

\medskip

Let $u_*$ be the minimizer of the variational problem \eqref{eq:43}.
Let $\mathcal B:=\{\beta>0:\L^1(\{x:|u_*(x)''|=\beta\})=0\}$. Clearly $\mathcal B$ contains all positive reals except a set of measure 0.
For   $\beta\in\mathcal B$ we  set
\[
  \begin{split}
I_\beta&:=\{x\in I:|u_*''(x)|> \beta\}\\
\mathcal I_\beta^n&:=\{i\in \{0,\dots,n-1\}: \max_{{X_i^n}} |u_*''|\geq\beta\}\,.
\end{split}
\]

\begin{lemma}
\label{lem:length_estimate}
 Let  $\beta\in\mathcal B$, and  $(u_n,y_n)$   such that $\bar y_n\in \PAM^n(I)$, $u_n\circ \bar y_n \in \PA^n(I)$, and $g_n:=n(u_n-u_*)$.  Then there exists a monotone increasing $\omega\in C^0([0,\infty))$ with $\omega(0)=0$ and $\omega(t)>0$ for $t>0$ such that
  \begin{equation*}\label{eq:9}
        \int_{I} |g_n'(x)|^2\d x\gtrsim \sum_{i\in\mathcal I_\beta^n}\omega(\beta)n^2\L^1(X_i^n)^3
=\omega(\beta)|b-a|^2\int_{ I_\beta^n} | y_n'(x)|^{-2}\d x\,.
  \end{equation*}
\end{lemma}
\begin{proof}
By the uniform continuity of $u_*''$, there exists $ \eta\equiv\eta(\beta)>0$ such that $|u_*''(x)-u_*''(x')|<\beta/2$ for $|x-x'|< \eta$. For $i\in \mathcal I_\beta^n$, choose $x_0$ such that  $|u_*''(x_0)|\geq \beta$.  Letting $J$ denote the intersection of $X_i^n$ with $[x_0-\eta,x_0+ \eta]$, we have that
\[
\inf_{v\in\R^n}\int_{J}|v-u_*'(x)|^2\d x\gtrsim \beta^2\min(\L^1(X_i^n),\eta)^3\,.
\]
Hence 
  \begin{equation}\label{eq:17}
\int_{J} |g_n'(x)|^2\d x=n^2\int_{J} |u_n'(x_i^n)-u_*'(x)|^2\d x\gtrsim n^2\beta^2 \min(\L^1(X_i^n),\eta)^3\gtrsim \eta^3\beta^2n^2\L^1(X_i^n)^3\,,
\end{equation}
which implies  the  first inequality in the statement of the present lemma by setting $\omega=\eta(\beta)^3\beta^2$ and summing over all $i\in \mathcal I_\beta^n$. The second relation in the statement  follows from $I_\beta\subset \cup_{i\in \mathcal I_\beta^n} X_i^n$ and the fact
  \begin{equation*}\label{eq:18}
    \begin{split}
 \L^1(X_i^n)^3=\left(\frac{b-a}{n}|\bar y_n'(\bar x_i^n)|\right)^3&
= \int_{\bar X_i^n} \frac{(b-a)^2}{n^2}|\bar y_n'(\bar x)|^3\d \bar x\\
&=\frac{(b-a)^2}{n^2}\int_{X_i^n} \frac{1}{|y_n'(x)|^{2}}\d x\,.
\end{split}
\end{equation*}
\end{proof}

\medskip

By Taylor's theorem, there exists a continuous $\tilde R_i^n(z)$ with $\lim_{z\to 0} \tilde R_i^n(z)=0$ uniformly in $i,n$ such that for $x\in X_i^n$, 
  \begin{equation}\label{eq:27}
g_n'(x)= \underbrace{\fint_{X_i^n}g_n'(t)\d t}_{=:A_i^n}+n(x-x_i^n)\left(\underbrace{\fint_{X_i^n}  u_*''(t)\d t}_{=:B_i^n}+\tilde R_i^n(x-x_i^n)\right)\,.
\end{equation}
For $x\in I$, we write 
\[
R^n(x)=\sum_{i=0}^{n-1}\chi_{X_i^n}(x)\tilde R(x-x_i^n)\,.
\]
Furthermore we introduce the piecewise constant functions 
\[
A^{n}(x)=\sum_{i=0}^{n-1}\chi_{X_i^n}(x)A_i^n\,,\qquad
B^{n}(x)=\sum_{i=0}^{n-1}\chi_{X_i^n}(x)B_i^n
\]
and the piecewise affine functions
\[
\ell^n(x)=\sum_{i=0}^{n-1}n\chi_{X_i^n}(x)(x-x_i^n)\,.
\]

With this notation in place, we may  decompose $g_n'$ as 
\[
g_n'=A^n+\ell^n(B^n+R^n)\,,
\]
which yields the following decomposition of
$g_n'\otimes g_n'$:
  \begin{equation}\label{eq:30}
  \begin{split}
    g_n'\otimes g_n'&=A^{n}\otimes A^{n}
    +\ell^n\left(A^{n}\otimes (B^{n}+R^{n})+(B^{n}+R^{n})\otimes A^{n}\right)\\
                          &\quad +|\ell^n|^2 (B^{n}+R^{n})\otimes (B^{n}+R^{n})\,.
  \end{split}
\end{equation}

\begin{proposition}
\label{lem:length_to_bounds}
Let $\beta>0$, $I_\beta=\{x\in I:|u_*''(x)|>\beta\}$ 
and suppose that $g_n\wto g$ in $W^{1,2}_0(I)$. Then the sequence $(\ell^n)_{n\in\N}$ is bounded in $L^2(I_\beta)$, and 
  we have the convergences \begin{equation}\label{eq:7}
    \begin{split}
      B^{n}&\to u_*''\,, \quad R^{n}\to 0 \quad \text{ in }L^\infty(I_\beta;\R^N)\,, \\
A^{n}&\wto g'\quad\text{ in }L^2(I_\beta)\,.
    \end{split}
  \end{equation}
\end{proposition}
\begin{proof}
  The first  line of \eqref{eq:7} follows from $\L(X_i^n)\to 0$ for $n\to\infty$, for every $i\in \{0,\dots,n-1\}$ such that $X^n_i\cap I_\beta\neq \emptyset$ (see \eqref{eq:17}), and the uniform continuity of $u_*''$. The second line follows from $\L(X_i^n)\to 0$ and $g_n'\wto g'$ in $L^2$. To prove the  boundedness of $(\ell^n)_{n\in\N}$, we calculate
\[
\int_{X_i^n}(x-x_i^n)^2\d x=\frac{1}{12} \L^1(X_i^n)^3
\]
which implies 
\[
  \begin{split}
\limsup_{n\to\infty}\int_{I_\beta} |\ell^n|^2\d x&\lesssim \limsup_{n\to\infty}\sum_{i\in\mathcal I^n_\beta}n^2\int  (x-x_i^n)^2\d x\\
&\leq \limsup_{n\to\infty}\beta^{-2}\|g_n\|_{L^2(I)}^2\,,
\end{split}
\]
 where we have used Lemma \ref{lem:length_estimate}. 
\end{proof}

\section{Proof of compactness and lower bound}
\label{sec:proof-comp-lower}
\begin{proof}[Proof of Theorem \ref{thm:Fjgamma} (o) and (i)]
By our assumptions on $\mathcal L$, $\mathcal F$ is $C^2$ G\^ateaux differentiable, and Taylor's theorem yields
  \begin{equation}\label{eq:19}
    \begin{split}
      \mathcal F(u_n)-\mathcal F(u_*)&= \underbrace{\delta \mathcal F(u_*,u_n-u_*)}_{=0}\\
                                     &\quad+\int_0^1(1-\tau)\delta^2 \mathcal F(u_*+\tau(u_n-u_*),u_n-u_*)\d \tau\,.
    \end{split}
  \end{equation}
We will now use the notation $g_n(x):=n(u_n(x)-u_*(x))$ from Section \ref{sec:auxiliary-notation}. Additionally we set
\[
G_n(x)= \begin{pmatrix}
    g_n(x)\\ g_n'(x)\end{pmatrix}\,.
\]
We may then write
  \begin{equation}\label{eq:51}
  \begin{split}
n^2(\mathcal F(u_n)-\mathcal F(u_*))&= 
\int_0^1(1-\tau)\delta^2 \mathcal F\left(u_*,\frac{\tau}{n}g_n\right)\d \tau\\
&= \int_I\int_0^1(1-\tau)\nabla_{(z,p)}^2\mathcal L\left(x,u_*+\frac{\tau}{n}g_n,u_*'+\frac{\tau}{n}g_n'\right):G_n(x)\otimes G_n(x)
\d \tau\d x\,.
\end{split}
\end{equation}
By  assumption (A1), we obtain that
\[
n^2(\mathcal F(u_n)-\mathcal F(u_*))
\geq \theta\int_I|g_n'(x)|^2\d x\,,
\]
and hence $ g_n'$ is bounded in $L^2(I)$ since we assume $\sup_n \mathcal F_n(u_n,y_n)<+\infty$. We may pass to a subsequence such that $g_n\wto g$ in $W^{1,2}_0(I;\R^N)$, and $G_n\wto G$ in $L^2(I;\R^{2N})$. 
By $\int_{I}y_n'(x)\d x\leq 1$ and standard compactness results for $BV$ functions (see e.g.~\cite{MR1857292}), we may now pass to a subsequence $y_n$ and some $y\in \MBV(I)$ such that
\[
y_n\wsto y \text{ in } BV(I)\,.
\]
This completes the proof of (o). After passing to a suitable subseqence,  we may assume from now on that $\mathcal F_n(u_n,y_n)$ converges to $\liminf_{n\to\infty}\mathcal F_n(u_n,y_n)$.

\medskip

Recall $I_\beta=\{x\in I:|u_*''(x)|>\beta\}$ for $\beta>0$. 
We have that $g_n/n$ and $g_n'/n$ converge to 0 strongly in $L^2$. 
We note that the function 
\[
(x,v,w)\mapsto \int_0^1 (1-\tau)\nabla_{(z,p)}\mathcal L(x,u_*+\tau v,u_*'+\tau w)\d \tau
\]
is continuous in all of its variables, in particular Carath\'eodory. 
Hence, by a standard approximation argument (see e.g.~equation (3.31) in the proof of \cite[Theorem 3.23]{dacorogna2024introduction}), for every $\e>0$ there exists a subsequence (no relabeling) and a measureable set $I^{\e}_\beta\subset I_\beta$ such that $\L^1(I_\beta\setminus I^{\e}_\beta)\leq \e$ and 
  \begin{equation}\label{eq:53}
\int_{I_\beta^\e}\left|\int_0^1 (1-\tau)\nabla_{(z,p)}\mathcal L\left(x,u_*+\tau\frac{g_n}{n},u_*'+\tau\frac{g_n'}{n}\right)\d \tau-\frac12\nabla_{(z,p)}\mathcal L(x,u_*,u_*')\right|\d x<\e\,.
\end{equation}

\medskip

Let $\zeta\in C^0_c(I)$ have the following properties:
\begin{equation}\label{eq:zetadef}
\begin{split}
0&\leq\zeta\leq 1\\
\zeta&=0 \text{ on } I\setminus I_\beta^\e \\
\L^1(\{\zeta=1\})&\geq \L^1(I_\beta^\e)-\e\,.
\end{split}
\end{equation}

\medskip

We will now estimate
  \begin{equation}\label{eq:52}
  \begin{split}
\int_{I}& \zeta(x)\nabla^2_{(z,p)} \mathcal L(x,u_*,u_*'):G_n\otimes G_n\d x\\
&=\int_{I} \zeta(x)\Bigg(\nabla^2_{p} \mathcal L(x,u_*,u_*'):g_n'\otimes g_n'+2\nabla_{z}\nabla_{p} \mathcal L(x,u_*,u_*'):g_n\otimes g_n'\\
&\quad+\nabla^2_{z} \mathcal L(x,u_*,u_*'):g_n\otimes g_n\Bigg)\d x\,.
\end{split}
\end{equation}
By the strong convergence $g_n\to g$  and the weak convergence $g_n'\wto g'$ in $L^2$ we obtain easily
\[
  \begin{split}
\lim_{n\to\infty} &\int_{I} \zeta(x)\Bigg(2\nabla_{z}\nabla_{p} \mathcal L(x,u_*,u_*'):g_n\otimes g_n'+
\nabla^2_{z} \mathcal L(x,u_*,u_*'):g_n\otimes g_n\Bigg)\d x\\
&= \int_{I} \zeta(x)\Bigg(2\nabla_{z}\nabla_{p} \mathcal L(x,u_*,u_*'):g\otimes g'+
\nabla^2_{z} \mathcal L(x,u_*,u_*'):g\otimes g\Bigg)\d x\,.
\end{split}
\]
In order to analyze the missing term 
\[
\int_{I} \underbrace{\zeta(x)\nabla^2_{p} \mathcal L(x,u_*,u_*')}_{W(x)}:g_n'\otimes g_n'\d x\,,
\]
we use the decomposition \eqref{eq:30} of $g_n'\otimes g_n'$ from Section \ref{sec:auxiliary-notation}. First we consider the contribution of the first term on the right hand in \eqref{eq:30}, $A^n\otimes A^n$. By Lemma \ref{lem:length_to_bounds},
\[
A^{n}\wto g'\quad\text { in }L^2(I_\beta)\,.
\]
Observing that  $W\in C^0(I;\R^{N\times N})$ has  values in the positive definite matrices, $\supp W\subset I_\beta$, and using standard  lower semicontinuity results for convex integral functionals under weak convergence (see e.g.~\cite[Theorem 3.23]{dacorogna2024introduction}), we obtain that   
  \begin{equation}\label{eq:13}
\liminf_{n\to\infty}\int_{I} W(x): A^{n}\otimes A^{n}\d x
\geq\int_{I}W(x):g'(x)\otimes g'(x)\d x\,.
\end{equation}

Now we consider the term $\ell^n\left(A^{n}\otimes (B^{n}+R^{n})+(B^{n}+R^{n})\otimes A^{n}\right)$ on the right hand side in \eqref{eq:30}, again integrated against $W$. 
By Lemma \ref{lem:length_to_bounds},  $\supp W\subset I_\beta$, and H\"older's inequality, we have that 
\[
  \begin{split}
\left|\int_{I}\ell^nW:(A^n\otimes R^n+R^n\otimes A^n)\d x\right|
&\lesssim \|\ell^n\|_{L^2(I_\beta)}\|W\|_{L^\infty(I_\beta)}\|A^n\|_{L^2(I_\beta)}\|R_n\|_{L^\infty(I_\beta)}\\
&\to 0 \text{ as } n\to\infty
\end{split}
\]
and hence
  \begin{equation}
\label{eq:48}  \begin{split}
\lim_{n\to\infty}&\int_{I}\ell^n W:\left(A^{n}\otimes (B^{n}+R^{n})+(B^{n}+R^{n})\otimes A^{n}\right)\d x\\
&=\lim_{n\to\infty}\int_{I}\ell^n W:\left(A^{n}\otimes B^{n}+B^{n}\otimes A^{n}\right)\d x\,.
\end{split}
\end{equation}
Introducing the notation
\[
W^n(x)=\sum_{i=0}^{n-1}\chi_{X_i^n}(x)\fint_{X_i^n}W(t)\d t\,,
\]
and using the fact that $\|W-W^n\|_{L^\infty}\to 0$, we  obtain
that \eqref{eq:48}  is equal to  
\[
\lim_{n\to\infty}\int_{I}\ell^nW^n:\left(A^{n}\otimes B^{n}+B^{n}\otimes A^{n}\right)\d x\,.
\]
For every $i=0,\dots,n-1$, we have that
\[
\int_{X_i^n}\ell^nW^n:\left(A^{n}\otimes B^{n}+B^{n}\otimes A^{n}\right)\d x=0\,.
\]
Hence 
  \begin{equation}\label{eq:41}
\lim_{n\to\infty}\int_{I}\ell^nW:\left(A^{n}\otimes (B^{n}+R^{n})+(B^{n}+R^{n})\otimes A^{n}\right)\d x=0\,.
\end{equation}
Now we analyse the contribution of the term $|\ell^n|^2 (B^{n}+R^{n})\otimes (B^{n}+R^{n})$ on the right hand side in \eqref{eq:30}. Then, by the bounds on $\|\ell^n\|_{L^2},\|B^n\|_{L^\infty}$ and the fact $\|R^n\|_{L^\infty}\to 0$ obtained in Lemma \ref{lem:length_to_bounds}, and $\|W-W^n\|_{L^\infty}\to 0$ :
\begin{equation}\label{eq:29}
\begin{split}
\liminf_{n\to\infty}\int_{I_\beta} |\ell^n|^2W:(B^{n}+R^{n})\otimes (B^{n}+R^{n})\d x
&=\liminf_{n\to\infty}\int_{I_\beta} |\ell^n|^2W:B^{n}\otimes B^{n}\d x\\
&=\liminf_{n\to\infty}\int_{I_\beta} |\ell^n|^2W^n:B^{n}\otimes B^{n}\d x\,.
\end{split}
\end{equation}
For every $X_i^n$, we may treat $W^n,B^n$ as constants and obtain
\[
  \begin{split}
\int_{X_i^n}|\ell^n|^2W^n:B^{n}\otimes B^{n}\d x&=W^n:B^n\otimes B^n\int_{X_i^n}n^2 (x-x_i^n)^2\d x\\
&=W^n:B^n\otimes B^n\int_{ X_i^n}\frac{(b-a)^2}{12}\frac{1}{y_n'^2}\d  x\,.
\end{split}
\]
Summing over all $i=0,\dots, n-1$, and using the strong convergences $B_n\to u_*''$, $W^n\to W$ in $L^\infty$, we get
  \begin{equation}\label{eq:42}
  \begin{split}
\liminf_{n\to\infty}&\int_{I}|\ell^n|^2 W:(B^{n}+R^{n})\otimes (B^{n}+R^{n})\d x\\
&=\frac{(b-a)^2}{12}\liminf_{n\to\infty}\int_{I} W^n:B^n\otimes B^n \frac{1}{y_n'^2}\d  x\\
&=\frac{(b-a)^2}{12}\liminf_{n\to\infty}\int_{I} W:u_*''\otimes u_*'' \frac{1}{y_n'^2}\d  x\,.
\end{split}
\end{equation}
By Lemma \ref{lem:weak_conv_conv}, we obtain 
  \begin{equation}\label{eq:37}
    \liminf_{n\to\infty}\int_{I} W:u_*''\otimes u_*'' \frac{1}{y_n'^2}\d  x\geq
    \int_{I} W:u_*''\otimes u_*'' \left|\frac{\d y'}{\d\L^1}\right|^{-2}\d  x \,.
  \end{equation}
Recalling the definition of $\zeta$ in \eqref{eq:zetadef}, and sending $\e,\beta$ to 0, $\zeta\equiv \zeta(\e,\beta)$ can be chosen such as to converge monotonously in $L^1$ to the characteristic function of $I^*:=I\setminus (u_*'')^{-1}(\{0\})$. Hence we obtain by \eqref{eq:53}, \eqref{eq:52},   \eqref{eq:42}, \eqref{eq:37}, and the monotone convergence theorem that
\[
  \begin{split}
 \int_{I^*}&\nabla^2_{(z,p)}\mathcal L: G\otimes G+
\frac{(b-a)^2}{12}\nabla_p^2\mathcal L(x,u_*,u_*'):u_*''\otimes u_*'' \left|\frac{\d y'}{\d\L^1}\right|^{-2}\d x\\
&\leq \liminf_{n\to\infty} \int_{I^*} \nabla^2_{(z,p)}\mathcal L: G_n\otimes G_n\d x\,.
\end{split}
\]
It remains to estimate the part of the integral on $I\setminus I^*$.
Here we can directly use the strong convergence $n^{-1}g_n,n^{-1}g_n'\to 0$ in $L^2$
 and the weak convergence $g_n\wto g_n$ in $W^{1,2}_0(I)$ to call upon \cite[Theorem 3.23]{dacorogna2024introduction} once more to obtain
\[
  \begin{split}
    \liminf_{n\to\infty}&\int_{I\setminus I^*} \int_0^1(1-\tau)\nabla_{(z,p)}^2\mathcal L\left(x,u_*+\tau\frac{g_n}{n},u_*'+\tau\frac{g_n'}{n}\right): G_n\otimes G_n \d \tau\d x\\
 &\geq \frac12 \int_{I\setminus I^*} \mathcal L(x,u_*,u_*'): G\otimes G\d x\,.
  \end{split}
\]
Summarizing, we get
\[
  \begin{split}
\frac12\delta^2\mathcal F(u_*,g)&+\frac{(b-a)^2}{24}\int_I\nabla^2_p\mathcal L(x,u_*,u_*'):u_*''\otimes u_*'' \left|\frac{\d y'}{\d\L^1}\right|^{-2}\d x\\
 &=\int_{I}\frac12\nabla^2_{(z,p)}\mathcal L: G\otimes G+
\frac{(b-a)^2}{24}\nabla^2_p\mathcal L:u_*''\otimes u_*'' \left|\frac{\d y'}{\d\L^1}\right|^{-2}\d x\\
&\leq \frac12\liminf_{n\to\infty}
\int_{I}\nabla^2_{(z,p)}\mathcal L: G_n\otimes G_n\d x\\
&=\liminf_{n\to\infty} \mathcal F_n(u_n,y_n)\,.
\end{split}
\]

\end{proof}

\section{Proof of the upper bound}
\label{sec:proof-upper-bound}

\begin{proof}[Proof of Theorem \ref{thm:Fjgamma} (ii)]
\emph{Step 1: Regularization of $y'$, definition of the recovery sequence for the regularized function.} Let  $g\in W^{1,2}_0(I)$ and $y\in \MBV(I)$ be as in the statement of the upper bound. We may assume $\mathcal F^*(g,y)<+\infty$, otherwise there is nothing to show. 
We write $\bar y=y^{-1}$. 
Let $\delta>0$ be a regularization parameter, and $Y^{(\delta)}$ be defined by $Y^{(\delta)}(a)=a$, and   \begin{equation}\label{eq:55}
    (Y^{(\delta)})'=\frac{1}{1+(b-a)\delta}\left(y'+\delta\L^1\right)\,,
  \end{equation}
For the steps 1-3 in the current proof, we will suppress the dependence on $\delta$ in the notation, to make it reappear in step 4 below.
As we have done before, we write $\bar Y=Y^{-1}$.
  For $n\in \N$, we  set
  \begin{equation}\label{eq:56}
    \bar Y_n\left(a+\frac{(b-a)i}{n}\right)=\bar Y\left(a+\frac{(b-a)i}{n}\right)\quad\text{ for }i=0,\dots,n\,,
  \end{equation}
define $\bar Y_n$ by affine interpolation on $\bar X_i^n$ (see \eqref{eq:23}), and write $Y_n=\bar Y_n^{-1}$ . This defines in particular $X_i^n=\bar Y_n(\bar X_i^n)$ . From our definition \eqref{eq:56} we obtain
\begin{equation}\label{eq:57}
  |Y_n'|^{-2}\to|Y'|^{-2} \quad\text{ in } L^1 \,, \qquad Y_n\wsto  Y \text{ in }BV(I)\,.
\end{equation}
From \eqref{eq:55}, we get 
  \begin{equation}\label{eq:49}
   n\left( \max_{i=0,\dots,n-1}\L^1(X_i^n)\right)\simeq \left\|\frac{1}{Y_n'}\right\|_{L^\infty}\simeq\|\ell^n\|_{L^\infty}\lesssim \delta^{-1}\,.
  \end{equation}
In particular, 
  \begin{equation}\label{eq:60}
\L^1(X_i^n)\to 0 \quad\text{ as } n\to\infty\,.
\end{equation}
Next we will define $u_n$ in several steps.   First, let 
\[
\alpha_n(t)=\begin{cases}\mathrm{sgn}(t) n &\text{ if } |t|>n\\
t &\text{ if }|t|\leq n\end{cases}
\]
and 
\[
\tilde g_n'(x):=\alpha_n(g'(x))-\fint_a^b\alpha_n(g'(t))\d t\,,\qquad
\tilde g_n(x)=\int_a^x \tilde g_n'(t)\d t\,.
\]
With these definitions we clearly have $\tilde g_n\to g$ in $W^{1,2}_0(I;\R^n)$, $\|\tilde g_n\|_{L^\infty}\leq C$,  $\|\tilde g_n'\|_{L^\infty}\leq n$ and 
\begin{equation}
\label{eq:gnstrong}
\lim_{n\to\infty}\left\|\left(\sum_{i=0}^{n-1}\chi_{X_i^n}\fint_{X_i^n}\tilde g'_n(t)\d t\right)-g'\right\|_{L^2(I)}=0\,.
\end{equation}
Then we set 
\[
  \begin{split}
u_n'(x)&=\fint_{X_i^n}\left(u_*'+\frac{\tilde g_n'}{n}\right)\d t \quad\text{ for } x\in X_i^n\,,\\
u_n(x)&=\int_a^x u_n'(t)\d t \quad\text{ for } x\in I\,.
\end{split}
\]
As before we set $g_n=n(u_n-u_*)$, which reads 
\[
g_n(x)=\fint_{X_i^n}\tilde g_n'(t)\d t +n\left(u_*'(x)-\fint_{X_i^n}u_*'(t)\d t\right)\quad\text{ for }x \in X_i^n\,.
\]
Clearly    $\|g_n\|_{L^\infty}\leq C$ and $\|g_n'\|_{L^\infty}\leq n$. By \eqref{eq:60}, we also have $g_n\wto g$ in $W^{1,2}_0(I;\R^N)$.

\medskip

With the same notation as in the previous section, 
\[
\mathcal F_n(u_n,y_n)
=
\int_I\int_0^1(1-\tau)\nabla_{(z,p)}^2\mathcal L\left(x,u_*+\frac{\tau}{n}g_n,u_*'+\frac{\tau}{n}g_n'\right):G_n(x)\otimes G_n(x)
\d \tau\d x\,.
\]

\medskip

\emph{Step 2: Equiintegrability, partial passing to the limit.} We claim that $G_n\otimes G_n$ is equiintegrable. Indeed, $|g_ng_n'|$ and $|g_n|^2$ are weakly converging in $L^1$, and hence equiintegrable, by the strong convergence of $g_n$ in $L^2$ and the weak convergence of $g_n'$ in $L^2$. It remains to show equiintegrability of $|g_n'|^2$. 
We will once more use the notation $A^n,B^n,R^n,\ell^n$ from Section \ref{sec:auxiliary-notation}, satisfying $g_n'=A^n+\ell^n(B^n+R^n)$.
Recalling $A^n=\sum_{i=0}^{n-1}\chi_{X_i^n}\fint_{X_i^n}g_n'\d t$, the strong convergence \eqref{eq:gnstrong} also implies 
\begin{equation}\label{eq:Anstrong}
A^n\to g' \qquad \text{ in }L^2(I)\,.
\end{equation}
Now for $A\subset I$, 
\[
  \begin{split}
\limsup_{n\to\infty}&\int_A |g_n'|^2\d x\\
&=\limsup_{n\to\infty}\int_ A \left(|A^n|^2+\ell^n(A^n\otimes (B^n+R^n)+(B^n+R^n)\otimes A^n)+|\ell^n|^2|B^n+R^n|^2\right)\d x\\
&\lesssim \limsup_{n\to\infty}\int_A \left(|A^n|^2+|(Y_n')^{-1}A^nB^n|+|Y_n'|^{-2}|B^n|^2\right)\d x\,.
\end{split}
\]
From the strong convergence of $A^n$ in $L^2$, as well as the convergence of $B^n$ in $L^\infty$ and \eqref{eq:49} we deduce that $|g_n'|^2$ is majorized by a strongly convergent sequence in $L^1$, which implies in particular the equiintegrability of $|g_n'|^2$.
Let $ \e>0$. Again we appeal to \cite[Theorem 3.23]{dacorogna2024introduction} to obtain the existence of $I^\e\subset I$ such that $\L^1(I\setminus I_\e)<\e$ and 
\[
\begin{split}
    \int_{I^\e}\Bigg|&\Bigg(\int_0^1(1-\tau)\nabla^2_{(z,p)}\mathcal L\left(x,u_*+\frac{\tau}{n} g_n,u_*'+\frac{\tau}{n} g_n'\right)\d\tau\\
    &\quad-\frac12\nabla^2_{(z,p)}\mathcal L(x,u_*,u_*')\Bigg):G_n\otimes G_n \Bigg|\d x<\e\,.
\end{split}
\]
By equiintegrability of $G_n\otimes G_n$ and uniform boundedness of \[
x\mapsto\int_0^1(1-\tau)\nabla^2_{(z,p)}\mathcal L\left(x,u_*+\frac{\tau}{n} g_n,u_*'+\frac{\tau}{n} g_n'\right)\d\tau
\]
in $L^\infty$, we have that the product of these is equiintegrable as well, and hence
\[
\begin{split}
\int_{I\setminus I^\e}\Bigg|&\Bigg(\int_0^1(1-\tau)\nabla^2_{(z,p)}\mathcal L\left(x,u_*+\frac{\tau}{n} g_n,u_*'+\frac{\tau}{n} g_n'\right)\d \tau\\
&\quad-\frac12\nabla^2_{(z,p)}\mathcal L(x,u_*,u_*')\Bigg):G_n\otimes G_n \Bigg|\d x\to 0 
\end{split}
\]
as $\e\to 0$, uniformly in $n$. 
Sending $\e\to 0$, we see that
  \begin{equation}\label{eq:59}
    \limsup_{n\to\infty}\mathcal F_n(u_n,y_n)=
    \limsup_{n\to\infty}\frac12\int_I \nabla^2_{(z,p)}\mathcal L(x,u_*,u_*'):G_n\otimes G_n\d x\,.
  \end{equation}

\medskip

\emph{Step 3: Explicit calculation on the microscale.}
Now we write 
  \begin{equation*}
    \begin{split}
    \fint_{X_i^n}\nabla^2_p\mathcal L(t,u_*,u_*')\d t&=\mathcal L^n_i\\
\mathcal L^n(x)&=\sum_{i=0}^{n-1}\mathcal L_i^n \chi_{X_i^n}(x)\,.
\end{split}
\end{equation*}

By Lemma \ref{lem:length_to_bounds} and step 1,
  \begin{equation}\label{eq:2}
    \begin{split}
R^n&\to 0,\quad  B^n\to u_*'' \quad\text{ in }L^\infty(I)\,,\\
A^{n}&\to g' \text{ in }L^2(I;\R^N)\,.
\end{split}
\end{equation}
Additionally, 
  \begin{equation}\label{eq:3}
    \mathcal L^n\to \nabla_p^2\mathcal L(\cdot,u_*,u_*')\quad\text{ in }L^\infty(I;\R^{N\times N})\,.
  \end{equation}
We use  \eqref{eq:30}, which yields a decomposition of $\int_{I}\nabla_p^2\mathcal L:g_n'\otimes g_n'\d x$ into three terms: 
The first term in this decomposition reads
\[
\lim_{n\to\infty} \int_{I} \nabla^2_{p}\mathcal L(x,u_*,u_*'):A^n\otimes A^n\d x=
 \int_{I} \nabla^2_{p}\mathcal L(x,u_*,u_*'):g'\otimes g'\d x\,,
\]
where we have used the second line of \eqref{eq:2}.
The second contribution is
\begin{equation}\label{eq:54}
  \begin{split}
    \lim_{n\to\infty}& \int_{I}\ell^n \nabla^2_{p}\mathcal L(x,u_*,u_*'):\left(A^n\otimes (B^n+R^n)+(B^n+R^n)\otimes A^n\right)\d x\\
                     &= \lim_{n\to\infty} \int_{I}\ell^n \mathcal L^n:\left(A^n\otimes B^n+B^n\otimes A^n\right)\d x=0\,.
  \end{split}
\end{equation}
To obtain the second line from the first one above,  we have used the bounds from \eqref{eq:49}, \eqref{eq:2} and \eqref{eq:3}. 
The third term in the decomposition is
\[
  \begin{split}
    \lim_{n\to\infty} &\int_{I} |\ell^n|^2\nabla^2_{p}\mathcal L(x,u_*,u_*'):(B^n+R^n)\otimes(B^n+R^n)\d x\\
&=\lim_{n\to\infty} \int_{I}|\ell^n|^2 \mathcal L^n: B^n\otimes B^n\d x\\
&=\frac{(b-a)^2}{12}\lim_{n\to\infty}\int_{I}\mathcal L^n: B^n\otimes B^n \frac{1}{Y_n'^2}\d x\,,
\end{split}
\]
where we have  used \eqref{eq:49}, \eqref{eq:2} and \eqref{eq:3} to obtain the first equality, and to
obtain the second one, we have used the explicit integration of a quadratic function on $X_i^n$, 
\[
\int_{X^n_i}|\ell^n|^2\d x=\frac{(b-a)^3}{12 n} \frac{1}{Y_n'^3|_{X_i^n}}=\L^1(X_i^n)\frac{(b-a)^2}{12 } \frac{1}{Y_n'^2|_{X_i^n}}\,,
\]
and summed over $i$. 
 Again using \eqref{eq:2} and \eqref{eq:3}, and additionally \eqref{eq:57},  we obtain
\[
\lim_{n\to\infty}\int_I\mathcal L^n: B^n\otimes B^n \frac{1}{Y_n'^2}\d x
= \int_I \nabla_p^2\mathcal L(x,u_*,u_*'):u_*''\otimes u_*''\frac{1}{Y'^2}\d x\,.
\]
Putting everything together, we have  obtained
\begin{equation}\label{eq:58}
  \begin{split}
    \lim_{n\to\infty} \int_I \nabla^2_{p}\mathcal L(x,u_*,u_*'):g_n'\otimes g_n'\d x 
    &=\int_I \nabla^2_{p}\mathcal L(x,u_*,u_*'):g'\otimes g'\d x\\
    &\quad+\frac{(b-a)^2}{12}\int_I \nabla^2_{p}\mathcal L(x,u_*,u_*'):u_*''\otimes u_*''\frac{1}{Y'^2}\d x\,.
  \end{split}
\end{equation}
By the strong convergence $g_n\to g$ in $L^2$ and the weak convergence $g_n'\wto g_n$ in $L^2$ we obtain
\[
    \lim_{n\to\infty}\int_I2\nabla_z\nabla_p\mathcal L:g_n\otimes g_n'+\nabla_z^2 \mathcal L:
                       g_n\otimes g_n\d x
                     =
                       \int_I2\nabla_z\nabla_p\mathcal L:g\otimes g'+\nabla_z^2 \mathcal L:
                       g\otimes g\d x\,.
\]

Combining the latter with \eqref{eq:58} and recalling \eqref{eq:59} yields the  upper bound for the regularized function $Y$, 
  \begin{equation}\label{eq:38}
\lim_{n\to \infty} \mathcal F_n(u_n,Y_n)=\mathcal F^*(g,Y)\,.
\end{equation}

\medskip

\emph{Step 4: Choosing $\delta$ as a function of $n$.}
We now make visible the dependence of $Y$ and the sequences $g_n,Y_n$ constructed in steps 1-3 on the regularization parameter $\delta$ by denoting them by $Y^{(\delta)},Y_n^{(\delta)},g_n^{(\delta)}$.  Writing $W:=\nabla^2_{(z,p)}\mathcal L(x,u_*,u_*'):u_*''\otimes u_*''$, we may assume that $W\left|\frac{\d y'}{\d\L^1}\right|^{-2}$ is integrable, and that hence
\[
W |(Y^{(\delta_k)})'|^{-2}\to W\left|\frac{\d y'}{\d\L^1}\right|^{-2}\quad\text{ in } L^1\,.
\]    
On $\MBV(I)$, the $BV$ weak-* convergence is metrizable (see e.g.~\cite[Lemma 1.4.1]{albiac2006topics}); let us denote a metric on this set by $d$. 
Let $(\delta_k)_{k\in\N}$ be a decreasing null sequence. Choose a strictly increasing sequence $(n_k)_{k\in\N}$ such that for every $n\geq n_k$,
\[
  \begin{split}
d(Y_{n}^{(\delta_{k})},Y^{(\delta_{k})})+\|g^{(\delta_k)}-g_{n}^{(\delta_k)}\|_{W^{1,2}}&\leq d(Y^{(\delta_{k})},y)\\
\text{ and }\quad|\mathcal F_{n}(u^{(\delta_{k})}_{n},Y_{n}^{(\delta_{k})})-\mathcal F^*(u^{(\delta_k)},Y^{(\delta_{k})})|&\leq |\mathcal F^*(g,y)-\mathcal F^*(g^{(\delta_k)},Y^{(\delta_{k})})|\,.
\end{split}
\]
The latter inequality is fulfilled for $n_k$ large enough by \eqref{eq:38}.
Now for $n\in\N$, define $\delta_n$ by $\delta_n=\delta_{n_k}$ for $n_k\leq n<n_{k+1}$. 
With this choice, $Y_{n}^{(\delta_{n})}\wsto y$ in $BV(I)$, $g^{(\delta_{n})}_{n}\to g$ in $W^{1,2}_0(I)$, and 
\[
\lim_{n\to\infty\to\infty}\mathcal F_{n}(u_{n}^{(\delta_{n})},Y_{n}^{(\delta_{n})})=\mathcal F^*(g,y)\,,
\]
proving the upper bound for $y_n:=Y_n^{(\delta_n)}$ and $u_n:=u_n^{(\delta_n)}$.
\end{proof}

\section{Numerical experiments}\label{sec:examples}

\subsection{Application of AMF}

We fix our domain to be $I=[a,b]=[0,1]$ and consider the Lagrangian
  \begin{equation}\label{eq:47}
  \begin{split}
\mathcal L: I\times\R\times\R&\to\R\\
  (x,z,p)&\mapsto \frac12|p|^2+f(x)\cdot z\,,
\end{split}
\end{equation}
i.e., the integrand of the Dirichlet energy plus a forcing term. We consider the minimization problem over the set of functions with zero boundary conditions,  $\mathcal A(U_0,U_1)=W^{1,2}_0(I)$. The unique minimizer will be denoted by $u_*$, and satisfies the Euler-Lagrange equation $u_*'' = f(x)$.
This 
yields in the limit, according to \eqref{eq:application}:
  \begin{equation}\label{eq:44}
\mathcal{F}^*(g, y) = \frac12\int_I |g'|^2\d x 
+ \frac{1}{24} \int_I |u_*''|^2 \, |y'(x)|^{-2} \, \mathrm{d}x\,,
\end{equation}
where $y\in \MBV(I)$ and $g\in W^{1,2}_0(I)$. 
We will consider the numerical minimization of this functional. Clearly the variables $g,y$ are decoupled, and the minimization in $g$ is trivial: It is given by $g=0$, which we assume from now on.
The  Euler--Lagrange equation associated to \eqref{eq:44} with $g=0$ leads to the optimality condition for $y$:
\[
f(x)^2 \, |y'(x)|^{-3} = \lambda_0,
\]
for some constant \( \lambda_0 > 0 \). Solving for \( y'(x) \) gives:
\[
y'(x) = \left( \frac{f(x)^2}{\lambda_0} \right)^{1/3} = \frac{f(x)^{2/3}}{\lambda_0^{1/3}}.
\]
Integrating over \( x \in [0,1] \) and applying the boundary conditions \( y(0) = 0 \) and \( y(1) = 1 \), we obtain the normalized solution:
  \begin{equation}\label{eq:45}
y(x) = \frac{\displaystyle \int_0^x |f(s)|^{2/3} \, \mathrm{d}s}{\displaystyle \int_0^1 |f(s)|^{2/3} \, \mathrm{d}s}.
\end{equation}
This expression defines the mapping \( y : [0,1] \to [0,1] \) that minimizes the functional \( Y\mapsto \mathcal{F}^*(0, Y) \), and adapts the mesh node distribution according to the target density \( f \). The normalization ensures consistency with the prescribed boundary conditions.

\begin{figure}[h!]
    \centering
    \includegraphics[width=0.49\linewidth]{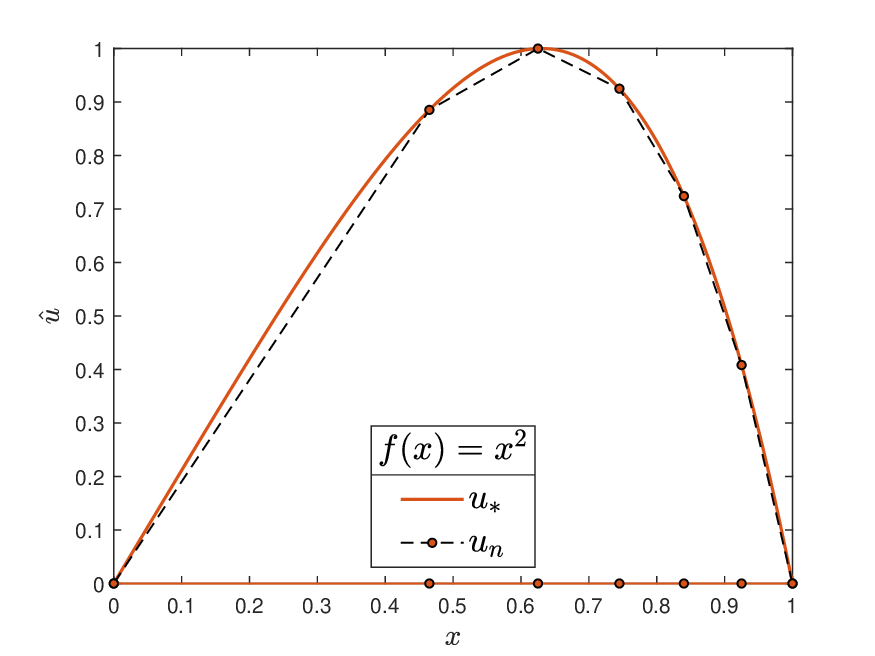}
        \includegraphics[width=0.49\linewidth]{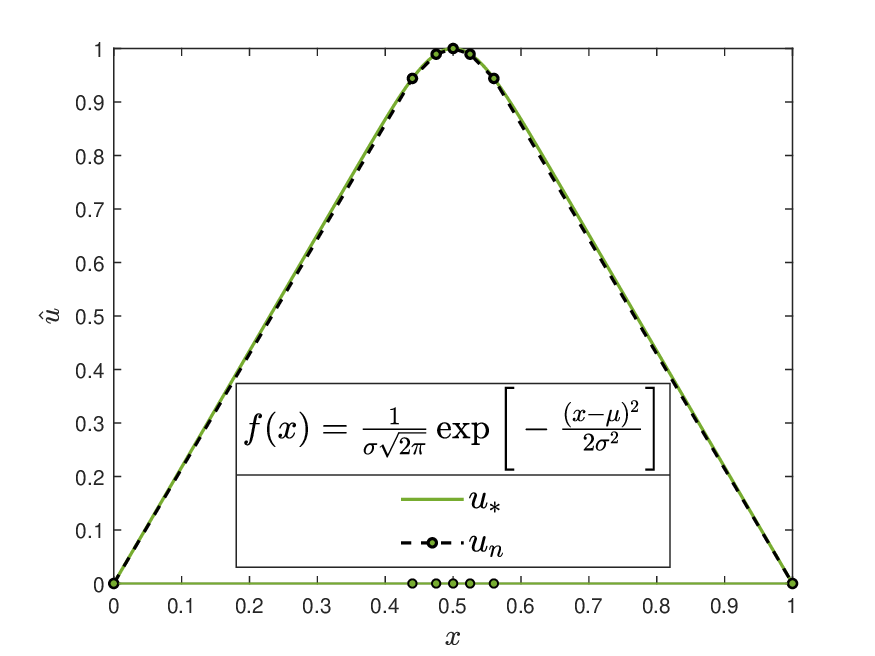}
    \caption{Optimal position of nodes for $f(x)=x^2$ and $f(x)=\frac{1}{\sigma\sqrt{2\pi}}\exp(  - \frac{(x-\mu)^2}{2\sigma^2})$, (with $\mu=0.5$ and $\sigma =0.05$). The dashed graph is the optimal piecewise affine function. For comparison, we display the exact solution $u_*$ (in color).} 
        \label{fig:fig2}
\end{figure}

\bigskip

Given the analytical expression of the mapping $y$ in the continuous limit (infinite number of nodes), it is natural to investigate its applicability in a discrete setting. An approximation of the optimal mesh  consisting of $n$ elements may be obtained by defining the elements via
  \begin{equation}\label{eq:46}
  X_i^n:=y(\bar X_i^n)=y([i/n,(i+1)/n))\,, \quad i=0,\dots,n-1\,.
\end{equation}
For later reference, we state this approach as an algorithm to find the optimal mesh with $n$ elements, that we label ``AMF'' for ``Asymptotic mesh functional'':

\begin{algo}[AMF]
\label{algo:AMF}
  \begin{enumerate}
  \item Compute the asymptotic optimal mesh via \eqref{eq:45}.
\item Determine the approximate optimal mesh with $n$ elements via \eqref{eq:46}.
  \end{enumerate}
\end{algo}

Figure~\ref{fig:fig2} illustrates the thusly obtained meshes for $n = 6$ elements, for the two cases $f(x)=x^2$ and $f(x)=\exp(-\pi|x-1/2|^2)$.


 %







\subsection{Comparison of the AMF algorithm with gradient descent}

We will now compare Algorithm \ref{algo:AMF} with a different approximation, namely the one obtained by conceiving of the node positions as additional variables and minimizing the energy via gradient descent. These variables are inherent to our notation $\mathcal F_n(u,y)$. Denoting the nodes in the reference and in the deformed configuration by 
\[
\bar\xi_i^n:=\frac{i}{n}\,,\quad \xi_i^n:=y^{-1}(\bar\xi_i^n)\,,
\]
respectively, we find that $\mathcal F_n(u,y)$  only depends on the $n-1$ node positions $\xi_i^n$, $i=1,\dots,n-1$, 
and the values of $u$ at $\xi_i^n$, $u_i^n\equiv u(\xi_i^n)$, $i=1,\dots,n-1$, where we assume that the value of the functional is finite.
We may write $\boldsymbol{\xi}^n:=(\xi_1^n,\dots,\xi_{n-1}^n)$, $\mathbf{u}^n:=(u_1^n,\dots,u_{n-1}^n)$,  and 
\[
\mathcal E_n(\boldsymbol{\xi}^n,\mathbf{u}^n):=\mathcal F_n(y,u)\,.
\]


For definiteness, we state the iterative gradient descent (GD) for $\mathcal E_n$:
\begin{algo}[GD]
  \label{algo:GD}
  \begin{enumerate}
  \item Initiate $(\boldsymbol{\xi}^{n,(0)},\mathbf{u}^{n,(0)})$, set $k=0$.
\item While 
\[
\| \nabla \mathcal E_n\|_1< 10^{-6}
\]
set 
\[
(\boldsymbol{\xi}^{n,(k+1)},\mathbf{u}^{n,(k+1)}) = (\boldsymbol{\xi}^{n,(k)},\mathbf{u}^{n,(k)}) - \eta \nabla \mathcal E_n( \boldsymbol{\xi}^{n,(k)},\mathbf{u}^{n,(k)})  \quad \text{with } \eta > 0,
\]
and $k\leftarrow k+1$.
  \end{enumerate}
\end{algo}


\newcommand{\bu}{\mathbf{u}}
\newcommand{\bx}{\boldsymbol{\xi}}

The efficiency of GD is of course highly dependent on the choice of the initialization. If the latter is chosen far away from the optimum, the algorithm may become highly inefficient. The function $(\bx,\bu)\mapsto \mathcal E_n(\bx,\bu)$ is non-convex, and hence the algorithm might get stuck in local minima. For this reason, we study it only as an improvement of AMF. I.e., the initialization step of GD will be given by AMF. In this way, we will get an impression of the quality of algorithm \ref{algo:GD}.

Figure~\ref{fig:u_equi_DG_AMF} presents the relative errors between the exact solution  and the corresponding finite element approximation in $L^2$ and $W^{1,2}_0$ respectively,
\[
  \begin{split}
\| u_* - u_n \|_{L^2}^R &= \frac{\| u_* - u_n \|_{L^2}}{\| u_* \|_{L^2}}  \\
\| u_* - u_n \|_{W^{1,2}_0}^R &= \frac{\| u_*' - u_n' \|_{L^2}}{\| u'_* \|_{L^2}}\,.
\end{split}
\]
These errors are compared for different numbers of nodes, and for different approximation schemes: AMF, DG, and for comparison, equidistributed finite elements (no optimization over interval lengths).  The relative $L^2$ error associated with AMF is bounded above by that of equidistributed finite elements, and below by that of the gradient-descent-optimized mesh, while remaining close to the latter.


\begin{figure}[h!]
    \centering
    \includegraphics[width=0.49\linewidth]{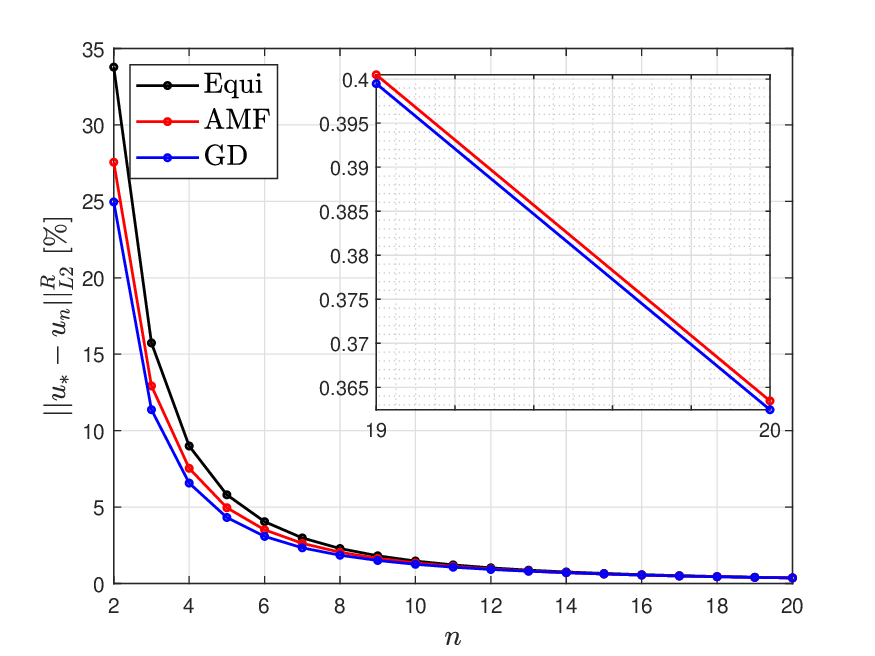}
        \includegraphics[width=0.49\linewidth]{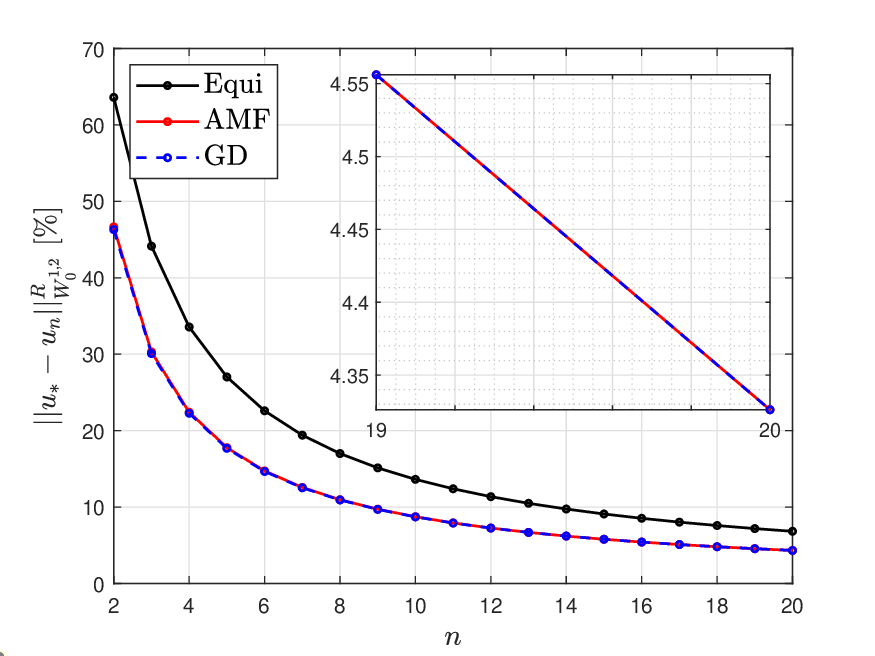}
    \caption{Relative error of the solution $u$ and $u'$ using an equal distributed mesh (black) with AMF (red), and descent gradient meshing method (blue).} 
    \label{fig:u_equi_DG_AMF}
\end{figure}

\begin{figure}[b!]
    \centering
        \includegraphics[width=0.49\linewidth]{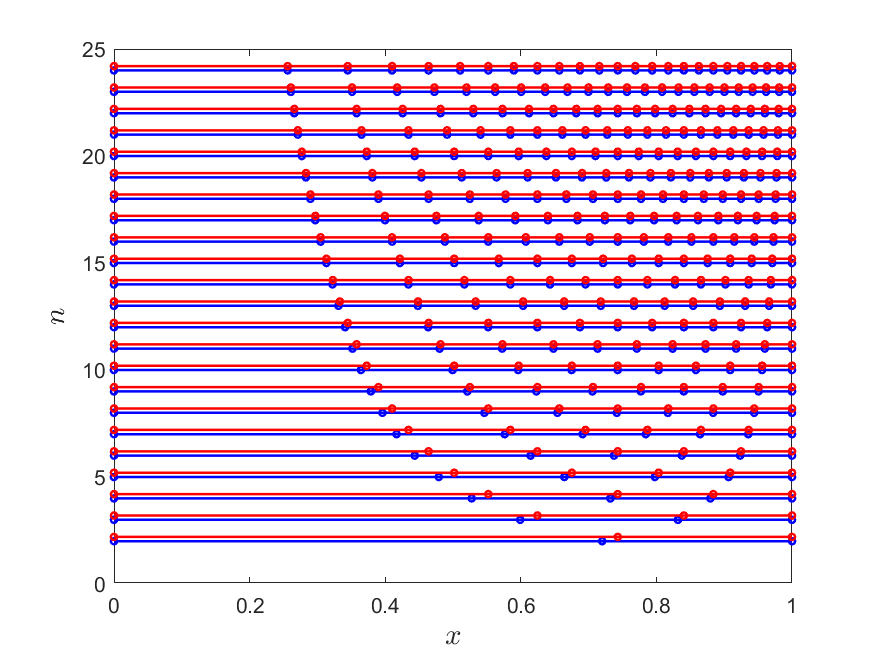}
        \includegraphics[width=0.49\linewidth]{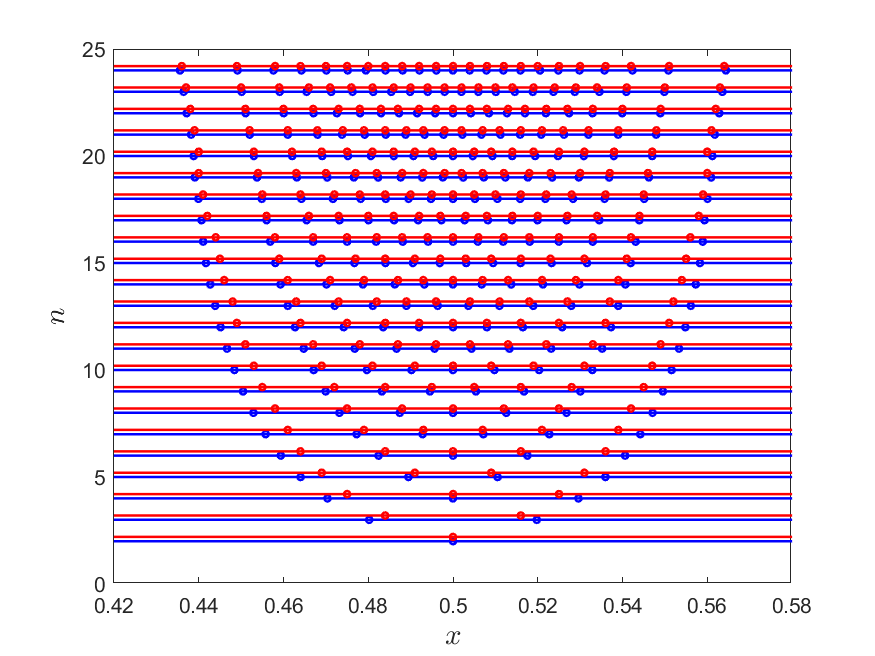}
    \caption{Mesh distribution error between GD and AMF method for different f (left) quadratic and (right) gaussian ($\mu=0.5,\sigma=0.03$).} 
    \label{fig:node_cinematic}
\end{figure}

Figure~\ref{fig:node_cinematic} displays the node positions for the meshes generated using the AMF and GD methods. A qualitative agreement between the two configurations is observed, with an increasingly accurate overlap as the number of nodes increases. In the following, the goal is to determine how far the AMF mesh is different compared to the optimized reference (GD) mesh.

In order to do so, we denote the minimizer  found by Algorithm \ref{algo:AMF} by $\bx_{AMF}\in \PAM^n(I)$ and the improvement obtained in Algorithm \ref{algo:GD} by $\bx_{GD}\in \PAM^n(I)$.

\begin{figure}[h!]
    \centering
    \includegraphics[width=0.49\linewidth]{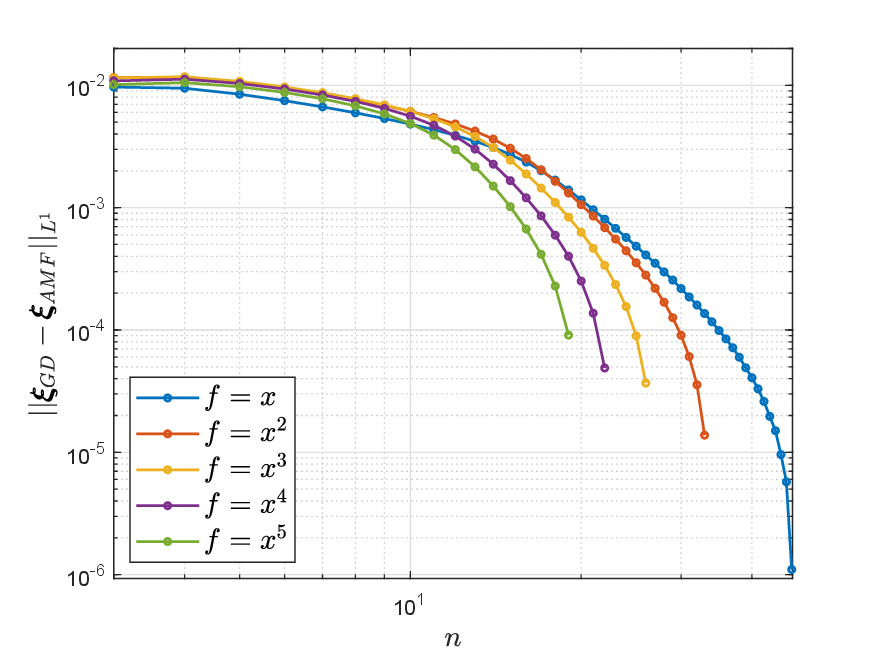}
    \includegraphics[width=0.49\linewidth]{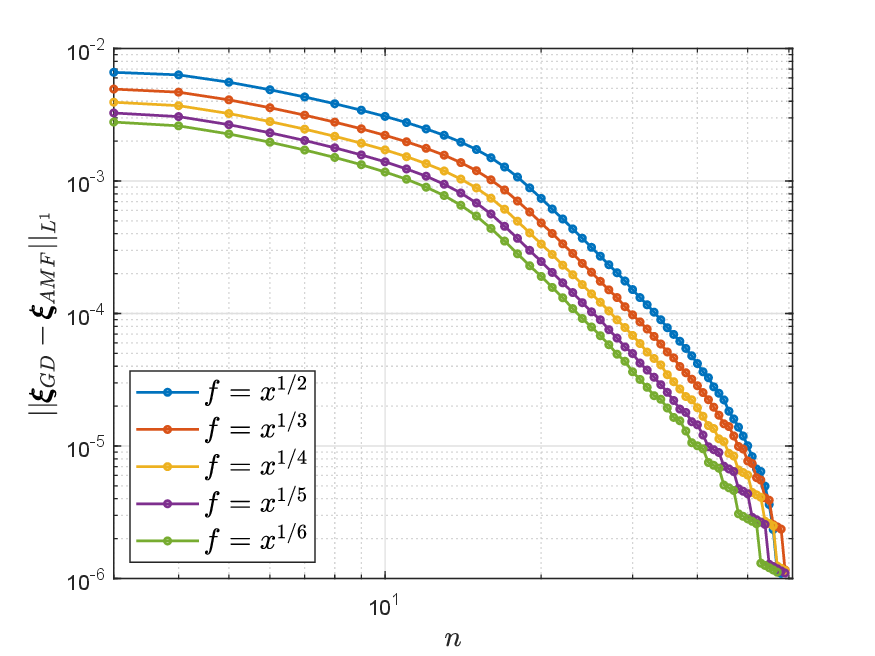}
    \caption{$L^1$-error between position of nodes obtained by  GD and AMF method for different rational functions.} 
    \label{fig:fig_f_polynom}
\end{figure}
Figure~\ref{fig:fig_f_polynom} shows the  $L^2$ error $\|\bx_{GD}-\bx_{AMF}\|_{L^2}$ for polynomial forcing terms of degree \( k \in \{1, 2, \ldots, 5\} \), as well as for root-type functions of order \( p \in \{2, 3, \ldots, 6\} \). For this class of functions, the discrepancy in node positions between the meshes generated by the AMF and GD methods decreases significantly and smoothly as the number of nodes increases.

\begin{figure}[h!]
    \centering
    \includegraphics[width=0.49\linewidth]{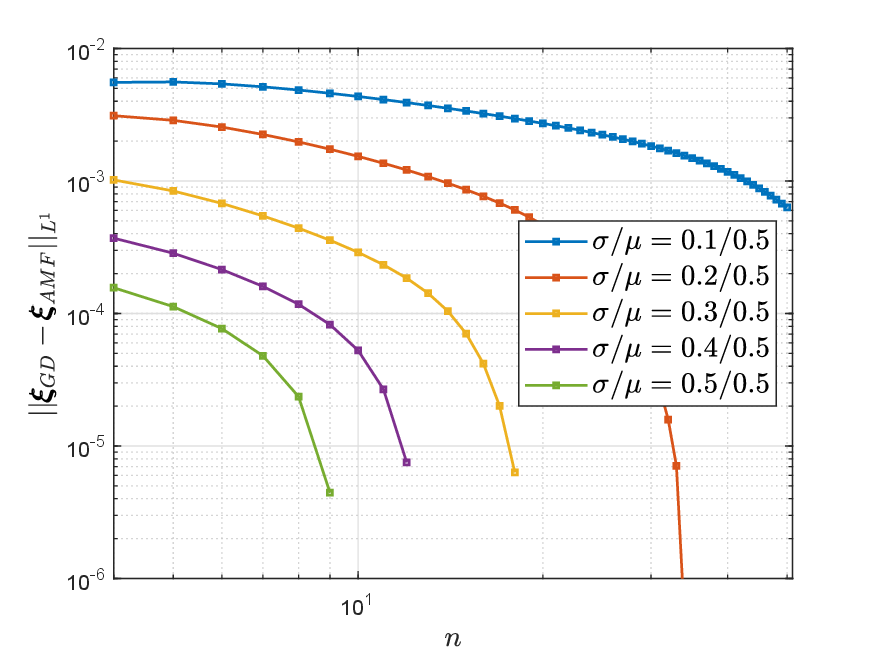}
    \caption{$L^1$-error between position of nodes obtained by  GD and AMF method for Gaussian functions with different variances.} 
    \label{fig:fig_sin_gauss}
\end{figure}

Figure~\ref{fig:fig_sin_gauss} shows the  $L^2$ error for  a parametrized Gaussian profile \( f = \Phi(\mu,\sigma;x) \) centered at \( \mu = 0.5 \) with different standard deviations \( \sigma \in \{0.1, 0.2, 0.3, 0.4, 0.5\} \). 

\medskip

For the first three class of functions, a quasi-monotonic decrease in the relative \( L^2 \)-error is observed as the number of nodes increases. This expected behavior reflects the effectiveness of the AMF method in approaching the optimal distribution obtained with the GD method as the mesh resolution increases. The quantitative comparison of the curves also allows for assessing the relative efficiency of both approaches depending on the nature of the forcing term.

\subsection*{Acknowledgements}
The authors thank Hidde Sch\"onberger for carefully proofreading the manuscript and for  helpful suggestions.

\appendix

\section{An auxiliary lemma}
\label{sec:aux_lem}

\begin{lemma}
\label{lem:weak_conv_conv}
  Let $I\subset\R$ be an interval, $(\mu_j)_{j\in\N}\subset \mathcal M(I)$ a sequence of (non-negative) measures converging narrowly, $\mu_j\wsto \mu\in\mathcal M(I)$. Let $f:[0,\infty)\to[0,\infty]$ be convex and decreasing such that $f^{-1}(+\infty)$ is closed, and $w\in C^0_c(I;[0,\infty))$. Then    
  \begin{equation}\label{eq:26}
    \liminf_{j\in\N}\int_I w\,f\left(\frac{\d \mu_j}{\d\L^1}\right)\d \L^1 \geq \int_I w\,f\left(\frac{\d \mu}{\d\L^1}\right)\d \L^1\,.
  \end{equation}
\end{lemma}

\begin{proof}
We may assume $\liminf_{j\in\N}\int_I w\,f\left(\frac{\d \mu_j}{\d\L^1}\right)\d \L^1<\infty$, which implies in particular $f(\d\mu/\d\L^1)<+\infty$ $w\L^1$-almost everywhere. After passing to a subsequence, we may suppose that the $\liminf$ is actually a limit. After passing to a further subsequence, we may assume that there exists a non-negative $\tilde f\in \mathcal M(I)$  with
\[
 w\,f\left(\frac{\d \mu_j}{\d\L^1}\right)\L^1\wsto \tilde f\,.
\]
Let us decompose $\mu,\tilde f$ into their Lebesgue regular and singular parts, 
\newcommand{\ac}{{\mathrm{ac}}}
\[
    \mu=\frac{\d\mu}{\d\L^1} \L^1+ \mu_s\,,\qquad
    \tilde f=\frac{\d\tilde f}{\d\L^1} \L^1+ \tilde f_s\,.
\]
For almost every $x_0$, we have that 
\[
    \lim_{r\to 0} \frac{1}{2r}\int_{B(x_0,r)} \d\tilde f = \frac{\d\tilde f}{\d\L^1}(x_0)\,,\qquad    \lim_{r\to 0} \frac{1}{2r}\int_{B(x_0,r)} \d\mu =\frac{\d\mu}{\d\L^1}(x_0)\,.
\]

In order to show \eqref{eq:26}, it suffices to show 
  \begin{equation}\label{eq:33}
    w(x_0) f\left(\frac{\d\mu}{\d\L^1}(x_0)\right)\leq \frac{\d\tilde f}{\d\L^1}(x_0)\,
  \end{equation}
for $w\L^1$-almost every $x_0$. In order to prove the latter, let $\e>0$. Choose $r>0$ such that $\mu(\partial B(x_0,r))=\tilde f(\partial B(x_0,r))=0$ and
\[
\frac{\d\tilde f}{\d\L^1}(x_0)> \frac{1}{2r}\int_{B(x_0,r)} \d\tilde f-\e\,,
\]
Then choose $j_0$ such that for $j>j_0$,
  \begin{equation}\label{eq:36}
    \begin{split}
      \frac{1}{2r}\int_{B(x_0,r)} \d\tilde f&\geq \fint_{B(x_0,r)} w\,f\left(\frac{\d\mu_j}{\d\L^1}\right)\d \L^1-\e \\
      w(x_0)f\left(\frac{1}{2r}\int_{B(x_0,r)}\d\mu_j\right)&\geq w(x_0) f\left(\frac{\d\mu}{\d\L^1}(x_0)\right)-\e\,.
    \end{split}
  \end{equation}
  The latter inequality can be achieved since we may assume that $f$ is finite and hence continuous in a neighborhood of $\d\mu/\dL^1(x_0)$.
  By possibly decreasing $r$, we have in addition to the previous relations
  \[
  \fint_{B(x_0,r)} w\,f\left(\frac{\d\mu_j}{\d\L^1}\right)\d \L^1\geq w(x_0) \fint_{B(x_0,r)} f\left(\frac{\d\mu_j}{\d\L^1}\right)\d \L^1-\e\,.
  \]
By Jensens's inequality,  
\[
\fint_{B(x_0,r)} f\left(\frac{\d\mu_j}{\d\L^1}\right)\d \L^1\geq f\left(\fint_{B(x_0,r)}\frac{\d\mu_j}{\d\L^1}\d\L^1\right)\,.
\]
By the monotonicity of $f$, 
\[
  \begin{split}
    f\left(\fint_{B(x_0,r)}\frac{\d\mu_j}{\d\L^1}\d\L^1\right)&\geq 
    f\left(\fint_{B(x_0,r)}\frac{\d\mu_j}{\d\L^1}\d\L^1+\frac{1}{2r}\int_{B(x_0,r)} \d(\mu_j)_s\right)\\
&=f\left(\frac{1}{2r}\int_{B(x_0,r)}\d\mu_j\right)\,,
  \end{split}
\]
where $(\mu_j)_s$ is the Lebesgue singular part of $\mu_j$, and $j>j_0$.
Putting all of the above together, we obtain 
\[
w(x_0) f\left(\frac{\d\mu}{\d\L^1}(x_0)\right)-4\e< \frac{\d\tilde f}{\d\L^1}(x_0) \,,
\]
which proves \eqref{eq:33} in the limit $\e\to 0$.
\end{proof}

\bibliographystyle{alpha}
\bibliography{grid}

\end{document}